\documentclass[11pt]{amsart}

\usepackage{float}

\usepackage{bbm}
\usepackage{amsfonts}
\usepackage{amsmath}
\usepackage{amssymb}
\usepackage{graphicx}
\usepackage[left=2.5cm, right=2.5cm, top=3cm]{geometry}

\usepackage{mathtools}
\usepackage{amsthm}
\usepackage{latexsym}
\usepackage{fancyhdr}
\usepackage{array}
\usepackage{amscd}
\usepackage{lscape}
\usepackage{tikz}
\usepackage{tikz-cd}
\usepackage{float}
\usepackage{bm}
\usepackage{subfigure}
\usepackage{grffile}
\usepackage{array}
\usepackage{longtable}
\usepackage{booktabs}
\usepackage{adjustbox}
\usepackage{diagbox}
\usepackage{algorithm}
 \usepackage{algpseudocode}
 \usepackage{algorithmicx}
 \algdef{SE}[DOWHILE]{Do}{doWhile}{\algorithmicdo}[1]{\algorithmicwhile\ #1}%

\algrenewcommand\algorithmicrequire{\textbf{Input:}}
\algrenewcommand\algorithmicensure{\textbf{Output:}}

\usepackage{ dsfont }
\usepackage{lipsum}
\newcolumntype{C}[1]{>{\centering\arraybackslash}p{#1}}

\definecolor{navy}{HTML}{2F729C} 
\usepackage[hyperfootnotes=false, colorlinks, linkcolor={blue}, citecolor={magenta}, filecolor={blue}, urlcolor={blue}, plainpages=false, pdfpagelabels]{hyperref}
\usepackage[tableposition=above]{caption}
\newcommand{\p}{\mathfrak p}

\newtheorem{theorem}{Theorem}[section]
\newtheorem{lemma}[theorem]{Lemma}
\newtheorem{proposition}[theorem]{Proposition}
\newtheorem{corollary}[theorem]{Corollary}
\theoremstyle{definition}

\newtheorem{mainthms}{Theorem}
\newtheorem{maincor}[mainthms]{Corollary}
\newtheorem{maincon}[mainthms]{Conjecture}

\numberwithin{equation}{section}

\begin{document}

\title[Lower bounds for the modified Szpiro ratio]{Lower bounds for the modified Szpiro ratio}

\author{Alexander J. Barrios}
\address{Department of Mathematics, University of St. Thomas, St. Paul, MN 55105 USA}
\email{abarrios@stthomas.edu}

\date{}

\begin{abstract}
Let $E/\mathbb{Q}$ be an elliptic curve. The modified Szpiro ratio of $E$ is the quantity~$\sigma_{m}\!\left(  E\right)  =\log\max\!\left\{  \left\vert c_{4}^{3}\right\vert ,c_{6}^{2}\right\} /\log N_{E}$ where $c_{4}$ and $c_{6}$ are the invariants associated to a global minimal model of $E$, and $N_{E}$ denotes the conductor of $E$. In this article, we show that for each of the fifteen torsion subgroups $T$ allowed by Mazur's Torsion Theorem, there is a rational number $l_{T}$ such that if $T\hookrightarrow E\!\left(\mathbb{Q}\right)  _{\text{tors}}$, then $\sigma_{m}\!\left(  E\right)  >l_{T}$. We also show that this bound is sharp.
\end{abstract}



\maketitle

\section{Introduction}

By an $ABC$ triple, we mean a triple $\left(  a,b,c\right)  $ where $a,b,c$ are relatively prime positive integers with $a+b=c$. The \textit{quality} of an $ABC$ triple $\left(  a,b,c\right)  $ is the quantity
\[
q\!\left(  a,b,c\right)  =\frac{\log c}{\log\!\left(  \operatorname{rad}%
\!\left(  abc\right)  \right)  }.
\]
Masser and Oesterl\'{e}'s $ABC$ Conjecture \cite{MR992208} states that for each $\epsilon>0$, there are finitely many $ABC$ triples with $q\!\left(  a,b,c\right)  >1+\epsilon$. In \cite[Theorem 4]{MR1635726}, Browkin, Filaseta, Greaves, and Schinzel showed that if the $ABC$ Conjecture holds, then the set of limit points of
\[
\left\{  q\!\left(  a,b,c\right)  \mid\left(  a,b,c\right)  \text{ is an
}ABC\text{ triple}\right\}
\]
is $\left[  \frac{1}{3},1\right]  $. In fact, they showed unconditionally that the set of limit points contains $\left[  \frac{1}{3},\frac{15}{16}\right]  $. In this article, we conjecture an analogous statement in the context of the modified Szpiro conjecture and establish a lower bound for the modified Szpiro ratio. To make this precise, we define the \textit{modified Szpiro ratio} of an elliptic curve $E$ over $\mathbb{Q}$ to be the quantity
\[
\sigma_{m}\!\left(  E\right)  =\frac{\log\!\left(  \max\!\left\{  \left\vert
c_{4}^{3}\right\vert ,c_{6}^{2}\right\}  \right)  }{\log N_{E}}%
\]
where $c_{4}$ and $c_{6}$ are the invariants associated to a global minimal model of $E$, and $N_{E}$ denotes the conductor of $E$. The modified Szpiro ratio is the analog of the quality of an $ABC$ triple. The $ABC$ Conjecture is equivalent to the modified Szpiro conjecture \cite{MR992208}, which states that for each $\epsilon>0$, there are finitely many elliptic curves (up to isomorphism) satisfying $\sigma_{m}\!\left(  E\right) >6+\epsilon$. By \cite{MR1635726}, we have that $\frac{1}{3}$ is a lower bound on the quality of an $ABC$ triple. In this article, we establish that there is a lower bound on $\sigma_{m}\!\left(  E\right)  $ which depends only on the torsion structure of~$E$. Specifically, we prove:

\begin{mainthms}\label{Thm1}\textit{Let $T$ be one of the fifteen torsion subgroups allowed by Mazur's
Torsion Theorem~\cite{MR488287}. If $E/\mathbb{Q}$ is an elliptic curve with $T\hookrightarrow E\!\left(\mathbb{Q}\right)_{\text{tors}}$, then $\sigma_{m}\!\left(  E\right)  >l_{T}$ where~$l_{T}$ is as given below. Moreover, the bounds are sharp.
\[
{\renewcommand{\arraystretch}{1.2}\renewcommand{\arraycolsep}{.132cm}%
\begin{array}
[c]{ccccccccc}\hline
T & C_{1} & C_{2} & C_{3},C_{2}\times C_{2} & C_{4} & C_{5},C_{6},C_{2}\times
C_{4} & C_{7},C_{8},C_{2}\times C_{6} & C_{9},C_{10} & C_{12},C_{2}\times
C_{8}\\\hline
l_{T} & 1 & 1.5 & 2 & 2.4 & 3 & 4 & 4.5 & 4.8\\\hline
\end{array}
}%
\]
}
\end{mainthms}

\noindent The lower bounds are proven in Theorem \ref{mainthm}. Theorem \ref{thmshp} shows that the bounds are sharp.
We note that parts of the proof rely on computer verification, which was done using Mathematica \cite{Mathematica} and SageMath \cite{sagemath}. The Mathematica and SageMath notebooks used for this paper can be found in~\cite{GitHubMSR}.

The bound $\sigma_{m}\!\left(  E\right)  >1$ is a direct consequence from the fact that $\left\vert \Delta\right\vert <\max\!\left\{  \left\vert c_{4}^{3}\right\vert ,c_{6}^{2}\right\}  $, where~$\Delta$ denotes the minimal discriminant of $E$. Consequently, establishing Theorem~\ref{Thm1} is reduced to proving the result for elliptic curves with a non-trivial torsion point. Such curves are parameterizable~\cite{MR0434947}, and in this article we consider families of elliptic curves $E_T$ (see Table~\ref{ta:ETmodel}) which have the property that they parameterize all elliptic curves over $\mathbb{Q}$ that have a non-trivial torsion point (see Proposition~\ref{rationalmodels}). By \cite{2020arXiv200101016B} and \cite{BarRoy}, there are necessary and sufficient conditions on the parameters of the family $E_{T}$ to determine its global minimal model as well as the conductor exponent $f_{p}$ of $E_{T}$ at a given prime. This, in turn, allows us in Section~\ref{enh:condbound} to establish an upper bound $\delta_{T}$ on the conductor $N_{T}$ of $E_{T}$. In section \ref{naiveheight}, we show that $\left\vert \delta_{T}^{l_{T}}\right\vert <\max\!\left\{  \left\vert c_{4}^{3}\right\vert ,c_{6}^{2}\right\}  $. 

In Section \ref{lowerboundsec}, we prove the first part of Theorem~\ref{Thm1} by using these inequalities. We then show that the lower bounds are sharp by considering fifteen parameterized families $F_T$ (see the proof of Theorem~\ref{thmshp}) with the property that their modified Szpiro ratios tend to $l_{T}$. A crucial ingredient in this proof is a result for when a squarefree polynomial
$f\!\left(  x\right)  \in%
\mathbb{Z}
\!\left[  x\right]  $ with content~$1$ has infinitely many squarefree outputs.
The cases when $f\!\left(  x\right)  $ is quadratic or cubic are due to Nagel
\cite{MR3069398} and Erd\"{o}s \cite{MR0056635}, respectively. Booker and
Browning \cite{MR3533307} show that if $f\!\left(
x\right)  $ has the property that it has no fixed prime divisor and each of
the irreducible factors of $f\!\left(  x\right)  $ have degree at most $3$, then there are infinitely many integers $n$ for which $f(n)$ is squarefree. In
Lemma \ref{polysquarefree}, we remove the assumption of no fixed prime divisor.
This allows us to apply \cite{MR3533307} in our setting, thereby showing that
the lower bounds are sharp. We note that for a general $f\!\left(  x\right)
\in%
\mathbb{Z}
\!\left[  x\right]  $ with $\deg f\geq4$, one needs the $ABC$ Conjecture to
show that $f\!\left(  x\right)  $ has infinitely many squarefree outputs
\cite{MR1654759}.

An automatic consequence of Theorem~\ref{Thm1} is the following corollary:

\begin{maincor}
\label{enhMSRtrivial}\textit{Let $E/%
\mathbb{Q}
$ be an elliptic curve such that $\sigma_{m}\!\left(  E\right)  \leq1.5$. Then
$E\!\left(
\mathbb{Q}
\right)  _{\text{tors}}$ is trivial.}
\end{maincor}
\noindent We conclude this article by utilizing Corollary~\ref{enhMSRtrivial} to show that
there are infinitely many $n\in%
\mathbb{Z}
$ for which the rational elliptic curve $y^{2}+y=x^{3}+\left(  3n+1\right)  x$
has trivial torsion subgroup and global Tamagawa number equal to $1$ (see
Corollary \ref{lastcor}). The global Tamagawa number appears as a factor in
the leading term of the $L$-function of an elliptic curve in the Birch and
Swinnerton-Dyer Conjecture \cite{BSD}. Lorenzini \cite{MR2961846} showed that
there are only four rational elliptic curves with a torsion point of order
$n\geq4$ with global Tamagawa number equal to $1$. He also gave an infinite
family of elliptic curves with a $3$-torsion point and global Tamagawa number equal to
$1$. In \cite{BarRoy}, all rational elliptic curves with a non-trivial torsion
point that have global Tamagawa number equal to $1$ were explicitly
classified. In particular, an infinite family of elliptic curves with a
$2$-torsion point and global Tamagawa number $1$ was produced. By Corollary~\ref{enhMSRtrivial}, we now have an infinite family of elliptic curves with trivial
torsion subgroup and global Tamagawa number equal to $1$.

Lastly, we state a conjecture that is the counterpart of \cite[Theorem 4]{MR1635726} in the context of the modified Szpiro ratio. To this end, let $\mathcal{E}_{\mathbb{Q}}$ denote the collection of $\mathbb{Q}$-isomorphism classes of elliptic curves and let $\left[  E\right]  _{\mathbb{Q}}$ denote the $\mathbb{Q}$-isomorphism class of $E$. Then we can extend $\sigma_{m}$ to be defined on $\mathcal{E}_{\mathbb{Q}}$ by taking $\sigma_{m}\!\left(  \left[  E\right]  _{\mathbb{Q}}\right)  =\sigma_{m}\!\left(  E\right)  $. Since $\sigma_{m}$ depends on a global minimal model of $E$, this extension is well-defined. With this terminology, Theorem~\ref{Thm1} and \cite[Theorem 4]{MR1635726}, lead us to conjecture the following:

\begin{maincon}\textit{Let $T$ be one of the fifteen torsion subgroups allowed by
Mazur's Torsion Theorem \cite{MR488287}. If the $ABC$ Conjecture holds, then
$\left[  l_{T},6\right]  $ is the set of limit points of
\[
R_T=\left\{  \sigma_{m}\!\left(  \left[  E\right]  _{\mathbb{Q}}\right)  \mid\left[  E\right]  _{\mathbb{Q}}\in\mathcal{E}_{\mathbb{Q}},\ E\!\left(\mathbb{Q}\right)  _{\text{tors}}\cong T\right\}  .
\]}
\end{maincon}

\noindent The proof of Theorem~\ref{Thm1} shows that $l_{T}$ is a limit point of $R_T$. It has also been established that there are infinitely many non-isomorphic elliptic curves $E$ such that $\sigma_m(E)>6$ \cite{MR1065152,MR2931308,barrios2019constructive,Bar1}. In \cite{Bar1}, the author showed that for each of the fifteen torsion subgroups $T$ allowed by Mazur's Torsion Theorem, there are infinitely many non-isomorphic elliptic curves $E$ such that $\sigma_m(E)>6$ with $E(\mathbb{Q})_\text{tors}\cong T$. In particular, if the $ABC$ Conjecture holds, we have that $6$ is a limit point of $R_T$.

\subsection{Notation and terminology}

We start by recalling some facts about elliptic curves over~$\mathbb{Q}$. For
further details, see \cite{MR1312368,MR2514094}. An elliptic curve $E/\mathbb{Q}$ is given by an (affine) Weierstrass model
\begin{equation}
E:y^{2}+a_{1}xy+a_{3}y=x^{3}+a_{2}x^{2}+a_{4}x+a_{6} \label{ch:inintroweier}%
\end{equation}
with each $a_{j}\in%
\mathbb{Q}
$. From (\ref{ch:inintroweier}), we define%
\begin{equation}%
\begin{array}
[c]{l}%
c_{4}=a_{1}^{4}+8a_{1}^{2}a_{2}-24a_{1}a_{3}+16a_{2}^{2}-48a_{4},\\
c_{6}=-\left(  a_{1}^{2}+4a_{2}\right)  ^{3}+36\left(  a_{1}^{2}%
+4a_{2}\right)  \left(  2a_{4}+a_{1}a_{3}\right)  -216\left(  a_{3}^{2}%
+4a_{6}\right)  .
\end{array}
\label{basicformulas}%
\end{equation}
The quantities $c_{4}$ and $c_{6}$ are the \textit{invariants associated to the Weierstrass model} of $E$. The discriminant of $E$ is then defined as $\Delta_{E}=\frac{c_{4}^{3}-c_{6}^{2}}{1728}$. Each elliptic curve $E/\mathbb{Q}$ is $\mathbb{Q}$-isomorphic to a \textit{global minimal model} $E^{\text{min}}$ where $E^{\text{min}}$ is given by a Weierstrass model of the form (\ref{ch:inintroweier}) such that each $a_{j}\in\mathbb{Z}$ and its discriminant $\Delta_{E}^{\text{min}}$ satisfies
\[
\Delta_{E}^{\text{min}}=\min\!\left\{  \left\vert \Delta_{F}\right\vert \mid
F\text{ is }%
\mathbb{Q}
\text{-isomorphic to }E\text{ and }F\text{ is given by (\ref{ch:inintroweier}%
) with }a_{j}\in%
\mathbb{Z}
\right\}  .
\]
In what follows, let $c_{4}$ and $c_{6}$ be the invariants associated to $E^{\text{min}}$. We call $\max\!\left\{  \left\vert c_{4}^{3}\right\vert ,c_{6}^{2}\right\}  $ the \textit{naive height} of $E$ and note that this is slightly different than what appears in the literature~\cite{MR1745599}. Namely, we have dropped the logarithm from the standard definition. If a prime~$p$ divides $\gcd\!\left(c_{4},\Delta_{E}^{\text{min}}\right)  $, then $E$ is said to have \textit{additive reduction at }$p$. If this is not the case, then $E$ is said to be \textit{semistable} \textit{at }$p$\textit{. If }$E$ is semistable at each prime, then we say $E$ is \textit{semistable}. The conductor $N_{E}$ of $E$ is given by
\vspace{-.4em}
\[
N_{E}=\prod_{p|\Delta_{E}^{\text{min}}}p^{f_{p}}%
\]
where $f_{p}$ is a positive integer that is computed via Ogg's formula
\cite{MR207694}. Namely, $f_{p}=v_{p}\!\left(  \Delta_{E}^{\text{min}}\right)
-m_{p}+1$ where $m_{p}$ is the number of components (without counting multiplicities), defined over
$\overline{\mathbb{F}}_{p}$, of the special fiber of the regular reduced
proper minimal model $\mathcal{C}_{p}$ of $E$ over $\mathbb{Z}_{p}$ \cite{Neron1964}. We note that the \textit{local Tamagawa number
$c_{p}$ at $p$} is the number of components of the special fiber of
$\mathcal{C}_{p}$ that have multiplicity $1$ and are defined over
$\mathbb{F}_{p}$. The \textit{global Tamagawa number} $c$ is defined as
$c=\prod_{p}c_{p}$. We note that for each prime $p$, Tate's Algorithm
\cite{MR0393039} provides instructions to compute $m_{p}$ and $c_{p}$. In
fact, $f_{p}\leq 1$ if $E$ is semistable at $p$, and in general, we have that $f_{2}\leq8$, $f_{3}\leq5$, and $f_{p}\leq2$ for~$p\geq5$.

The group of rational points $E\!\left(\mathbb{Q}\right)  $ of an elliptic curve $E/\mathbb{Q}$ is a finitely generated abelian group and by Mazur's Torsion Theorem, there
are precisely fifteen possibilities for the torsion subgroup $E\!\left(\mathbb{Q}\right)  _{\text{tors}}$ of $E\!\left(\mathbb{Q}\right)  $.

\begin{theorem}
[Mazur's Torsion Theorem \cite{MR488287}]\label{MazurTorThm}Let $E/\mathbb{Q}$ be an elliptic curve and let $C_{n}$ denote the cyclic group of order $n$. Then
\[
E\!\left(  \mathbb{Q}\right)  _{\text{tors}}\cong\left\{
\begin{array}
[c]{ll}%
C_{n} & \text{for }n=1,2,\ldots,10,12,\\
C_{2}\times C_{2n} & \text{for }n=1,2,3,4.
\end{array}
\right.
\]

\end{theorem}

Now let $E_{T}$ be the parameterized family of elliptic curves given in Table~\ref{ta:ETmodel} for the listed~$T$. These families parameterize all rational elliptic curves with a non-trivial torsion point as made precise by the following proposition:

\begin{proposition}
[{\cite[Proposition 4.3]{2020arXiv200101016B}}]\label{rationalmodels}Let $E/%
\mathbb{Q}
$ be an elliptic curve and suppose further that $T\hookrightarrow E\!\left(
\mathbb{Q}
\right)  _{\text{tors}}$ where $T$ is one of the fourteen non-trivial torsion
subgroups allowed by Theorem~\ref{MazurTorThm}. Then there are integers
$a,b,d$ such that

$\left(  1\right)  $ If $T\neq C_{2},C_{3},C_{2}\times C_{2}$, then $E$ is
$\mathbb{Q}$-isomorphic to $E_{T}\!\left(  a,b\right)  $ with $\gcd\!\left(
a,b\right)  =1$ and $a$ is positive.

$\left(  2\right)  $ If $T=C_{2}$ and $C_{2}\times C_{2}\not \hookrightarrow
E(\mathbb{Q})$, then $E$ is $\mathbb{Q}$-isomorphic to $E_{T}\!\left(
a,b,d\right)  $ with $d\neq1,b\neq0$ such that $d$ and $\gcd\!\left(
a,b\right)  $ are positive squarefree integers.

$\left(  3\right)  $ If $T=C_{3}$ and the $j$-invariant of $E$ is not $0$,
then $E$ is $\mathbb{Q}$-isomorphic to $E_{T}\!\left(  a,b\right)  $ with
$\gcd\!\left(  a,b\right)  =1$ and $a$ is positive.

$\left(  4\right)  $ If $T=C_{3}$ and the $j$-invariant of $E$ is $0$, then
$E$ is either $\mathbb{Q}$-isomorphic to $E_{T}\!\left(  24,1\right)  $ or to
the curve $E_{C_{3}^{0}}\!\left(  a\right)  :y^{2}+ay=x^{3}$ for some positive
cubefree integer $a$.

$\left(  5\right)  $ If $T=C_{2}\times C_{2}$, then $E$ is $\mathbb{Q}%
$-isomorphic to $E_{T}\!\left(  a,b,d\right)  $ with $\gcd\!\left(
a,b\right)  =1$, $d$ positive squarefree, and $a$ is even.
\end{proposition}

\vspace{-0.85em} {\renewcommand*{\arraystretch}{1.18}
\begin{longtable}{C{0.6in}C{1in}C{1.5in}C{1.5in}C{0.5in}}
\caption[Weierstrass Model of $E_{T}$]{The Weierstrass Model $E_{T}:y^{2}+a_{1}xy+a_{3}y=x^{3}+a_{2}x^{2}+a_{4}x$}\\
\hline
$T$ & $a_{1}$ & $a_{2}$ & $a_{3}$ & $a_{4}$ \\
\hline
\endfirsthead
\caption[]{\emph{continued}}\\
\hline
$T$ & $a_{1}$ & $a_{2}$ & $a_{3}$ & $a_{4}$\\
\hline
\endhead
\hline
\multicolumn{2}{r}{\emph{continued on next page}}
\endfoot
\hline
\endlastfoot
$C_{2}$ & $0$ & $2a$ & $0$ & $a^{2}-b^{2}d$ \\\hline
$C_{3}^{0}$ & $0$ & $0$ & $a$ & $0$ \\\hline
$C_{3}$ & $a$ & $0$ & $a^{2}b$ & $0$ \\\hline
$C_{4}$ & $a$ & $-ab$ & $-a^{2}b$ & $0$ \\\hline
$C_{5}$ & $a-b$ & $-ab$ & $-a^{2}b$ & $0$ \\\hline
$C_{6}$ & $a-b$ & $-ab-b^{2}$ & $-a^{2}b-ab^{2}$ & $0$\\\hline
$C_{7}$ & $a^{2}+ab-b^{2}$ & $a^{2}b^{2}-ab^{3}$ & $a^{4}b^{2}-a^{3}b^{3}$ & $0$ \\\hline
$C_{8}$ & $-a^{2}+4ab-2b^{2}$ & $-a^{2}b^{2}+3ab^{3}-2b^{4}$ & $-a^{3}b^{3}+3a^{2}
b^{4}-2ab^{5}$ & $0$ \\\hline
$C_{9}$ & $a^{3}+ab^{2}-b^{3}$ & $
a^{4}b^{2}-2a^{3}b^{3}+
2a^{2}b^{4}-ab^{5}
$ & $a^{3}\cdot a_{2}$
& $0$ \\\hline
$C_{10}$ &$
a^{3}-2a^{2}b-
2ab^{2}+2b^{3}
$ & $-a^{3}b^{3}+3a^{2}b^{4}-2ab^{5}$ & $(a^{3}-3a^{2}b+ab^{2})\cdot a_{2}$ & $0$\\\hline
$C_{12}$ & $
-a^{4}+2a^{3}b+2a^{2}b^{2}-
8ab^{3}+6b^{4}
$ & $b(a-2b)(a-b)^{2}(a^{2}-3ab+3b^{2})(a^{2}-2ab+2b^{2})
$ & $a(b-a)^3 \cdot a_{2} $& $0$ \\\hline
$C_{2}\times C_{2}$ & $0$ & $ad+bd$ & $0$ & $abd^{2}$ \\\hline
$C_{2}\times C_{4}$ & $a$ & $-ab-4b^{2}$ & $-a^{2}b-4ab^{2}$ & $0$ \\\hline
$C_{2}\times C_{6}$ & $-19a^{2}+2ab+b^{2}$ & $
-10a^{4}+22a^{3}b-
14a^{2}b^{2}+2ab^{3}
$ & $
90a^{6}-198a^{5}b+116a^{4}b^{2}+
4a^{3}b^{3}-14a^{2}b^{4}+2ab^{5}
$ & $0$ \\\hline
$C_{2}\times C_{8}$ & $
-a^{4}-8a^{3}b-
24a^{2}b^{2}+64b^{4}$ & $-4ab^{2}(a+2b)(a+4b)^{2}(a^{2}+4ab+8b^{2}) $ & $ -2b(a+4b)(a^{2}-8b^{2}) \cdot a_{2}
$ & $0$
\label{ta:ETmodel}	
\end{longtable}}

\section{Upper Bound on the Conductor of
\texorpdfstring{$E_{T}$}{}\label{enh:condbound}}

Let $E_{T}$ be as defined in Table~\ref{ta:ETmodel} for $T\neq C_{3}^{0}$. Throughout this section, we let
\[
N_{T}=\left\{
\begin{array}
[c]{ll}%
N_{T}\!\left(  a,b,d\right)  & \text{if }T=C_{2},C_{2}\times C_{2},\\
N_{T}\!\left(  a,b\right)  & \text{if }T\not =C_{2},C_{2}\times C_{2}%
\end{array}
\right.
\]
be the conductor of $E_{T}$. In this section, we establish an upper bound on
$N_{T}$ which will allow us to deduce our main result. To this end, let%
\[
\gamma_{T}=\left\{
\begin{array}
[c]{ll}%
\gamma_{T}\!\left(  a,b,d\right)  & \text{if }T=C_{2},C_{2}\times C_{2},\\
\gamma_{T}\!\left(  a,b\right)  & \text{if }T\not =C_{2},C_{2}\times C_{2}.
\end{array}
\right.
\]
be as defined in \cite[Table 6]{2020arXiv200101016B}. Now write%
\begin{equation}
a=\left\{
\begin{array}
[c]{ll}%
c^{3}d^{2}e\text{ with }\gcd\!\left(  d,e\right)  =1\ \text{and }de\text{
squarefree} & \text{if }T=C_{3},\\
c^{2}d\text{ with }d\text{ squarefree} & \text{if }T=C_{4}.
\end{array}
\right.  \label{valueofaC34}%
\end{equation}
By \cite[Theorem 4.4]{2020arXiv200101016B}, the minimal discriminant of
$E_{T}$ is $u_{T}^{-12}\gamma_{T}$ where%
\[
{\renewcommand{\arraystretch}{1.2}\renewcommand{\arraycolsep}{.147cm}%
\begin{array}
[c]{cccccccc}\hline
T & C_{5},C_{7},C_{9} & C_{6},C_{8},C_{10},C_{12},C_{2}\times C_{2} &
C_{2},C_{2}\times C_{4} & C_{2}\times C_{6} & C_{2}\times C_{8} & C_{3} &
C_{4}\\\hline
u_{T} & 1 & 1\ \text{or }2 & 1,2,\ \text{or\ }4 & 1,4,\ \text{or\ }16 &
1,16,\ \text{or\ }64 & c^{2}d & c\ \text{or\ }2c\\\hline
\end{array}
}%
\]
In fact, \cite[Theorem 4.4]{2020arXiv200101016B} provides necessary and
sufficient conditions on the parameters of~$E_{T}$ to determine $u_{T}$. For
each $u_{T},$ let
\begin{equation}
\delta_{T,u_{T}}=\left\{
\begin{array}
[c]{ll}%
\delta_{T,u_{T}}\!\left(  a,b,d\right)  & \text{if }T=C_{2},C_{2}\times
C_{2},\\
\delta_{T,u_{T}}\!\left(  c,d,e,b\right)  & \text{if }T=C_{3},\\
\delta_{T,u_{T}}\!\left(  c,d,b\right)  & \text{if }T=C_{4},\\
\delta_{T,u_{T}}\!\left(  a,b\right)  & \text{if }T\neq C_{2},C_{3}%
,C_{4},C_{2}\times C_{2},
\end{array}
\right.  \label{delutdef}%
\end{equation}
be as given in Table \ref{ta:delT}. 
These expressions are also found in \cite[definitions.sage]{GitHubMSR}.
The main result of this section shows that if the minimal discriminant
of $E_{T}$ is $u_{T}^{-12}\gamma_{T}$, then
$N_{T}\leq\left\vert \delta_{T,u_{T}}\right\vert $. \vspace{-0.5em}
{\renewcommand*{\arraystretch}{1.2}\begin{longtable}{cC{2.8in}ccc}
\caption{The Polynomials $\delta_{T,u_{T}}$}
		\label{ta:delT}	\\
\hline
$T$ & $\delta_{T}$ & $u_{T}$, $\delta_{T,u_{T}}$ & $u_{T}$, $\delta_{T,u_{T}}$ & $u_{T}$, $\delta_{T,u_{T}}$\\
\hline
\endfirsthead
\caption[]{\emph{continued}}\\
\hline
$T$ & $\delta_{T}$ & $u_{T}$, $\delta_{T,u_{T}}$ & $u_{T}$, $\delta_{T,u_{T}}$ & $u_{T}$, $\delta_{T,u_{T}}$\\
\hline
\endhead
\hline
\multicolumn{3}{r}{\emph{continued on next page}}
\endfoot
\hline
\endlastfoot
	
$C_{2}$ & $b^{2}d(  b^{2}d-a^{2})  $ & $1$, $256\delta_{T}$ & $2$, $4\delta_{T}$ & $4$, $\frac{1}{64}\delta_{T}$\\\hline
$C_{3}$ & $3bd^{2}e^{4}(  c^{3}d^{2}e-27b)  $ & $c^{2}d$, $\delta_{T}$ & &\\\hline
$C_{4}$ & $bcd^{3}(  16b+c^{2}d)  $ & $c$, $2\delta_{T}$ & $2c$, $\frac{1}{16}\delta_{T}$ &  \\\hline
$C_{5}$ & $ab(  a^{2}+11ab-b^{2})  $ & $1$, $\delta_{T}$ & & \\\hline
$C_{6}$ & $ab(  a+b)  (  a+9b)  $ & $1$, $\delta_{T}$ & $2$, $\frac{1}{8}\delta_{T}$ &\\\hline
$C_{7}$ & $ab(  a-b)  (  a^{3}+5a^{2}b-8ab^{2}
+b^{3})  $ & $1$, $\delta_{T}$ & &\\\hline
$C_{8}$ & $ab(  a-2b)  (  a-b)  (
a^{2}-8ab+8b^{2})  $& $1$, $\delta_{T}$ & $2$, $\frac{1}{8}\delta_{T}$ &\\\hline
$C_{9}$ & $ab(  a-b)  (  a^{2}-ab+b^{2})  (
a^{3}+3a^{2}b-6ab^{2}+b^{3})  $ & $1$, $\delta_{T}$ & &\\\hline
$C_{10}$ & $ab(  a-2b)  (  a-b)  (
a^{2}+2ab-4b^{2})  (  a^{2}-3ab+b^{2})  $ & $1$, $\delta_{T}$ & $2$, $\frac{1}{4}\delta_{T}$ &\\\hline
$C_{12}$ & $ab(  a-2b)  (  a-b)  (
a^{2}-6ab+6b^{2})  (  a^{2}-2ab+2b^{2})  (
a^{2}-3ab+3b^{2})  $ & $1$, $\delta_{T}$ & $2$, $\frac{1}{8}\delta_{T}$ & \\\hline
$C_{2}\times C_{2}$ & $abd^{3}(  a-b)  $ & $1$, $64\delta_{T}$ & $2$, $\delta_{T}$ &\\\hline
$C_{2}\times C_{4}$ & $ab(  a+4b)  (  a+8b)
$ & $1$, $8\delta_{T}$ & $2$, $\frac{1}{2}\delta_{T}$  & $4$, $\frac{1}{32}\delta_{T}$\\\hline
$C_{2}\times C_{6}$ & $a(  a-b)  (  3a-b)  (
5a-b)  (  9a-b)  (  3a+b)   $ & $1$, $\delta_{T}$ & $4$, $\frac{1}{8}\delta_{T}$ & $16$, $\frac{1}{512}\delta_{T}$\\\hline
$C_{2}\times C_{8}$ & $ab(  a+2b)  (  a+4b)
(  a^{2}-8b^{2})  (  a^{2}+8ab+8b^{2})  (
a^{2}+4ab+8b^{2})  $ & $1$, $2\delta_{T}$ & $16$, $\frac{1}{128}\delta_{T}$ & $64$, $\frac{1}{4096}\delta_{T}$
\end{longtable}}

\begin{lemma}
\label{UpperBoundLemma}If $p$ is a prime that divides the minimal discriminant
$u_{T}^{-12}\gamma_{T}$ of $E_{T}$, then~$p$ divides $\delta_{T,u_{T}}$. In
particular, if $E_{T}$ is semistable at $p$, then $v_{p}\!\left(
N_{T}\right)  \leq v_{p}\!\left(  \delta_{T,u_{T}}\right)  $.
\end{lemma}

\begin{proof}
Throughout this proof, we implicitly assume the necessary and sufficient conditions on the parameters of $E_{T}$ found in \cite[Theorem 4.4]{2020arXiv200101016B} to determine the minimal discriminant $u_{T}^{-12}\gamma_{T}$ of $E_{T}$.

First, suppose $T=C_{3}$ so that $u_{T}=c^{2}d$. Recall that $a=c^{3}d^{2}e$ for $d$ and $e$ relatively prime squarefree integers. In particular, the minimal discriminant of $E_{T}$ is $u_{T}^{-12}\gamma_{T}=b^{3}d^{4}e^{8}\left(  c^{3}d^{2}e-27b\right)  $. By inspection, the lemma holds in this case. Next, suppose $T=C_{4}$. Then $u_{T}$ is either $c$ or $2c$. Recall that $a=c^{2}d$ for $d$ a squarefree integer. If $u_{T}=c$, then the minimal discriminant of $E_{T}$ is $u_{T}^{-12}\gamma_{T}=b^{4} c^{2}d^{7}\left(  c^{2}d+16b\right)  $. In particular, any prime dividing the minimal discriminant divides $\delta_{T,u_{T}}$. Now suppose $u_{T}=2c$ so that $v_{2}\!\left(  c\right)  \geq4$ with $bd\equiv3\ \operatorname{mod}4$. Then $\delta_{T,u_{T}}$ is even. Moreover, $u_{T}^{-12} \gamma_{T}=2^{-12}b^{4}c^{2}d^{7}\left(  c^{2}d+16b\right)  $ and so any odd prime dividing the minimal discriminant divides~$\delta_{T,u_{T}}$.

Now suppose $T\neq C_{3},C_{4}$. Then $u_{T}=2^{k}$ for some nonnegative
integer $k$. Consequently, if an odd prime $p$ divides the minimal
discriminant of $E_{T}$, then $p$ divides $\gamma_{T}$. It follows by
inspection that $p$ divides $\delta_{T,u_{T}}$. Furthermore, if $u_{T}=1$, then
$\gamma_{T}$ is the minimal discriminant and the lemma follows since
$\gamma_{T}$ being even implies that $\delta_{T,1}$ is even.

By the above, it remains to show the cases corresponding to $T\not =C_{3},C_{4},C_{5}%
,C_{7},C_{9}$ with~$u_{T}\neq1$ and $2$ dividing the minimal discriminant of
$E_{T}$.

\textbf{Case 1.} Suppose $T=C_{2}$ with $u_{T}\neq1$. Then $u_{T}$ is either
$2$ or $4$. Since $\delta_{T,u_{T}}$ is even if $u_{T}=2$, we have that the
lemma holds in this case. So suppose $u_{T}=4$ so that $v_{2}\!\left(
b^{2}d-a^{2}\right)  \geq8$ with $v_{2}\!\left(  a\right)  =v_{2}\!\left(
b\right)  =1$ and $a\equiv2\ \operatorname{mod}8$. In particular,
$\delta_{T,u_{T}}$ is always even, which concludes this case.

\textbf{Case 2.} Suppose $T=C_{6}$ with $u_{T}\neq1$. Then $u_{T}=2$ and
$v_{2}\!\left(  a+b\right)  \geq3$ and thus $\delta_{T,u_{T}}$ is even since
$a+9b$ is even.

\textbf{Case 3.} Suppose $T=C_{8},C_{10},C_{12}$ with $u_{T}\not =1$. Then
$u_{T}=2$ and $a$ is even in each case. By inspection, $\delta_{T,u_{T}}$ is even.

\textbf{Case 4.} Suppose $T=C_{2}\times C_{2}$ with $u_{T}\neq1$. Then
$u_{T}=2$ and $v_{2}\!\left(  a\right)  \geq4$ with $bd\equiv
1\ \operatorname{mod}4$. In particular, $\delta_{T,u_{T}}$ is even.

\textbf{Case 5. }Suppose $T=C_{2}\times C_{4}$ with $u_{T}\neq1$. Then
$u_{T}=2$ or $4$. In both cases, we have that $v_{2}\!\left(  a\right)  \geq
2$, and thus $\delta_{T,u_{T}}$ is even.

\textbf{Case 6.} Suppose $T=C_{2}\times C_{6}$ with $u_{T}\neq1$. Then
$u_{T}=4$ or $16$. If $u_{T}=4$, then $v_{2}\!\left(  a+b\right)  \geq2$ and
thus $\delta_{T,u_{T}}$ is even. Now suppose $u_{T}=16$ so that $v_{2}%
\!\left(  a+b\right)  =1$. In particular, $a+b\equiv2\ \operatorname{mod}4$,
and this implies that $a-b,\ 5a-b,$ and $9a-b$ are each divisible by $4$. In
fact, one of these must be divisible by $8$. This coupled with the fact that
$\left(  3a-b\right)  \left(  3a+b\right)  $ is divisible by $8$, implies that
$\delta_{T,u_{T}}$ is even since $v_{2}\!\left(  512\delta_{T,u_{T}}\right)
\geq10$.

\textbf{Case 7.} Suppose $T=C_{2}\times C_{8}$ with $u_{T}\neq1$. Then $u_{T}$
is either $16$ or $64$. If $u_{T}=16$, then $v_{2}\!\left(  a\right)  =1$ and
by inspection $\delta_{T,u_{T}}$ is even. So suppose $u_{T}=64$ so that
$v_{2}\!\left(  a\right)  \geq2$. By inspection $\delta_{T,u_{T}}$ is also
even under these assumptions, which concludes the proof.
\end{proof}


Next, we show that if $u_{T}^{-12}\gamma_{T}$ is the minimal discriminant of $E_{T}$, then
$N_{T}\leq\left\vert \delta_{T,u_{T}}\right\vert $. By
Lemma \ref{UpperBoundLemma}, it suffices to show that $v_{p}\!\left(
N_{T}\right)  \leq v_{p}\!\left(  \delta_{T,u_{T}}\right)  $ for each prime
$p$ for which $E_{T}$ has additive reduction. By \cite[Theorem 7.1]%
{2020arXiv200101016B}, there are necessary and sufficient conditions on the
parameters of $E_{T}$ to determine those primes $p$ for which $E_{T}$ has
additive reduction. For these primes $p$, \cite[Section 3]{BarRoy} gives
$v_{p}\!\left(  N_{T}\right)  $ in terms of the parameters of $E_T$. The following
result summarizes these conditions on the parameters, and as a consequence, we
obtain that $N_{T}\leq\left\vert \delta_{T,u_{T}}\right\vert $.

\begin{proposition}\label{ch:enh:bdcon}
Suppose that the parameters of $E_{T}$ satisfy the conclusion of Proposition~\ref{rationalmodels} and that $u_{T}^{-12}\gamma_{T}$ is the minimal discriminant of $E_T$. Then $E_{C_{2}\times C_{8}}$ is semistable and for
$T\neq C_{2}\times C_{8}$, $E_{T}$ has additive reduction at a prime $p$ if
and only if $p$ is listed in Table~\ref{ta:conddel}, and the corresponding condition on the parameters of $E_T$ is satisfied. For these conditions, the table also lists the following
information: the value(s) of $u_{T}$ and bounds on~$v_{p}\!\left(  N_{T}\right)$ and~$v_{p}\!\left(
\delta_{T,u_{T}}\right)  $.

In particular, if $u_{T}^{-12}\gamma_{T}$ is the minimal discriminant of $E_T$, then $N_{T}\leq\left\vert \delta_{T,u_{T}}\right\vert $.
\end{proposition}

{\renewcommand*{\arraystretch}{1.1}
\begin{longtable}[c]{cccccc}
\caption{Necessary and sufficient conditions on the parameters of $E_{T}$ to determine
additive reduction at a prime $p$. For the listed conditions, the value(s) of
$u_{T}$ for which $u_{T}^{-12}\gamma_{T}$ is the minimal discriminant 
of $E_{T}$ is given. Bounds on $v_{p}(N_{T})$ and $v_{p}(\delta_{T,u_{T}})$ are also
given for these conditions on the parameters.
}\\
\hline
$T$ & $p$ & Conditions on the parameters of $E_T$ & $v_{p}(N_{T})$ & $u_{T}$ & $v_{p}
(\delta_{T,u_{T}})$ \\
\hline
\endfirsthead
\caption[]{\emph{continued}}\\
\hline
$T$ & $p$ & Conditions on the parameters of $E_T$ & $v_{p}(N_{T})$ & $u_{T}$ & $v_{p}%
(\delta_{T,u_{T}})$\\
\hline
\endhead
\hline
\multicolumn{6}{r}{\emph{continued on next page}}
\endfoot
\hline
\endlastfoot
$C_{2}$ & $\neq2$ & $v_{p}(\gcd(a,bd))\geq1$ & $2$ & $1,2,4$ & $\geq2$\\\cmidrule(lr){2-6}
& $2$ & $v_{2}(b^{2}d-a^{2})\geq8$ with $a\equiv6\ \operatorname{mod}8$ & $4$
& $2$ & $\geq8$\\\cmidrule(lr){3-6}
&  & $4\leq v_{2}(b^{2}d-a^{2})\leq7$ with $v_{2}(a)=v_{2}(b)=1$ & $\leq5$ &
$2$ & $\geq8$\\\cmidrule(lr){3-6}
&  & $v_{2}(b^{2}d-a^{2})\leq3$ with $v_{2}(a)=v_{2}(b)=1$ & $7$ & $1$ &
$\geq8$\\\cmidrule(lr){3-6}
&  & $v_{2}(b)\geq3$ with $a\not \equiv 3\ \operatorname{mod}4$ or
$v_{2}(b)=0,2$ & $\leq8$ & $1$ & $\geq8$\\\cmidrule(lr){3-6}
&  & $v_{2}(b)=1$ with $v_{2}(a)\neq1$ & $\leq6$ & $1$ & $\geq8$\\\hline
$C_{3}$ & $\neq3$ & $v_{p}(a)\not \equiv 0\ \operatorname{mod}3$ & $2$ &
$c^{2}d$ & $\geq2$\\\cmidrule(lr){2-6}
& $3$ & $v_{3}(a)\geq1$ with $v_{3}(a)\not \equiv 0\ \operatorname{mod}3$ &
$\leq5$ & $c^{2}d$ & $\geq5$\\\cmidrule(lr){3-6}
&  & $v_{3}(a)\geq1$ with $v_{3}(a)\equiv0\ \operatorname{mod}3$ & $\leq4$ &
$c^{2}d$ & $\geq4$\\\hline
$C_{4}$ & $\neq2$ & $v_{p}(a)\equiv1\ \operatorname{mod}2$ & $2$ & $c,2c$ &
$\geq3$\\\cmidrule(lr){2-6}
& $2$ & $v_{2}(a)\equiv1\ \operatorname{mod}2$ with $v_{2}(a)\geq3$ & $\leq6$
& $c$ & $\geq6$\\\cmidrule(lr){3-6}
&  & $v_{2}(a)=4,6$ or $v_{2}(a)\geq8$ with $bd\equiv1\ \operatorname{mod}4$ &
$\leq5$ & $c$ & $\geq6$\\\cmidrule(lr){3-6}
&  & $v_{2}(a)=1,2$ & $3$ & $c$ & $\geq4$\\\hline
$C_{5}$ & $5$ & $v_{5}(a+3b)>0$ & $2$ & $1$ & $\geq2$\\\hline
$C_{6}$ & $2$ & $v_{2}(a+b)=1,2$ & $2$ & $1$ & $\geq2$\\\cmidrule(lr){2-6}
& $3$ & $v_{3}(a)>0$ & $2$ & $1,2$ & $\geq2$\\\hline
$C_{7}$ & $7$ & $v_{7}(a+4b)>0$ & $2$ & $1$ & $\geq2$\\\hline
$C_{8}$ & $2$ & $v_{2}(a)\geq2$ & $4$ & $1$ & $\geq4$\\\hline
$C_{9}$ & $3$ & $v_{3}(a+b)>0$ & $3$ & $1$ & $\geq 3$\\\hline
$C_{10}$ & $5$ & $v_{5}(a+b)>0$ & $2$ & $1,2$ & $\geq2$\\\hline
$C_{12}$ & $3$ & $v_{3}(a)>0$ & $2$ & $1,2$ & $\geq2$\\\hline
$C_{2}\times C_{2}$ & $\neq2$ & $v_{p}(d)=1$ & $2$ & $1,2$ & $\geq3$\\\cmidrule(lr){2-6}
& $2$  & $v_{2}(a)=1,2,3$ or $bd\not \equiv 1\ \operatorname{mod}4$ & $\leq6$ &
$1$ & $\geq7$\\\hline
$C_{2}\times C_{4}$ & $2$ & $v_{2}(a)=1$ & $3$ & $1$ & $\geq4$\\\cmidrule(lr){3-6}
&  & $v_{2}(a)\geq2$ with $v_{2}(a+4b)\leq3$ & $\leq4$ & $2$ & $\geq4$\\\hline
$C_{2}\times C_{6}$ & $3$ & $v_{3}(b)>0$ & $2$ & $1,4,16$ & $\geq3$%

\label{ta:conddel}	
\end{longtable}}
\begin{proof}
By \cite[Theorem 7.1]{2020arXiv200101016B}, we have necessary and sufficient
conditions on the parameters of~$E_{T}$ to determine those primes $p$ for
which $E_{T}$ has additive reduction. Moreover, Theorems~$3.2$, $3.3$, $3.5$,
$3.6$, $3.7$, and $3.8$ of \cite{BarRoy} provide necessary and sufficient
conditions to determine $v_{p}(N_{T})$ for those primes $p$ at which $E_{T}$
has additive reduction. These results verify the first four columns of Table
\ref{ta:conddel}. The fifth column is verified from the given conditions on
the parameters and \cite[Theorem 4.4]{2020arXiv200101016B}. 

We now show that the bound on $v_{p}\!\left(  \delta_{T,u_{T}}\right)  $ is as
given. To this end, let $u_{T}^{-12}\gamma_{T}$ denote the minimal
discriminant of $E_{T}$. We proceed by cases, and note that we will implicitly
assume the conditions listed in Table \ref{ta:conddel} that result in $E_{T}$
having additive reduction at a given prime~$p$.

\textbf{Case 1.} Suppose $T=C_{2}$. If $E_{T}$ has additive reduction at an
odd prime $p$, then $p$ divides $\gcd\!\left(  a,bd\right)  $. By inspection,
$v_{p}(\delta_{T,u_{T}})\geq2$ for each $u_{T}$. So suppose that $E_{T}$ has
additive reduction at $2$. If $u_{T}=1$, then $v_{2}\!\left(  \delta_{T,u_{T}%
}\right)  \geq8$. It remains to consider the cases corresponding to $u_{T}=2$.
Observe that in these cases, $v_{2}\!\left(  b^{2}d-a^{2}\right)  \geq4$ and
thus $v_{2}\!\left(  \delta_{T,u_{T}}\right)  \geq8$.

\textbf{Case 2.} Suppose $T=C_{3}$ and write $a=c^{3}d^{2}e$ for relatively
prime positive squarefree integers $d$ and $e$. If $E_{T}$ has additive
reduction at a prime $p\neq3$, then $v_{p}\!\left(  a\right)  \not \equiv
0\ \operatorname{mod}3$. Equivalently, $v_{p}\!\left(  de\right)  >0$ and hence $v_{p}\!\left(  \delta_{T,u_{T}}\right)  \geq2$. Now suppose that
$E_{T}$ has additive reduction at $3$. Then $v_3(a)>0$.
If $v_{3}\!\left(  a\right)
\not \equiv 0\ \operatorname{mod}3$, then $v_{3}\!\left(  de\right)  >0$. Now
observe that $v_{3}\!\left(  3d^{2}e^{4}\left(  c^{3}d^{2}e-27b\right)
\right)  \geq5$ and thus $v_{3}\!\left(  \delta_{T,u_{T}}\right)  \geq5$. If
instead $v_{3}\!\left(  a\right)
\equiv0\ \operatorname{mod}3$, then $v_{3}\!\left(  c\right)  \geq1$ and thus
$v_{3}\!\left(  3\left(  c^{3}d^{2}e-27b\right)  \right)  \geq4$. Consequently, $v_{3}\!\left(  \delta_{T,u_{T}}\right)  \geq4$.

\textbf{Case 3.} Suppose $T=C_{4}$ and write $a=c^{2}d$ for $d$ a positive
squarefree integer. If $E_{T}$ has additive reduction at a prime $p\neq2$,
then $v_{p}\!\left(  a\right)  $ is odd. Equivalently, $p$ divides $d$ and it
follows that $v_{p}\!\left(  \delta_{T,u_{T}}\right)  \geq3$ for each $u_{T}$.
So suppose that $E_{T}$ has additive reduction at $2$. Then $u_{T}=c$ and we
observe that $v_{2}\!\left(  \delta_{T,u_{T}}\right)  =1+v_{2}\!\left(
cd^{3}\right)  +v_{2}\!\left(  a+16b\right)  $. In particular, $v_{2}\!\left(
\delta_{T,u_{T}}\right)  \geq6$ if $v_{2}\!\left(  a\right)  \geq3$ and
$v_{2}\!\left(  \delta_{T,u_{T}}\right)  \geq4$ if $v_{2}\!\left(  a\right)
=1,2$.

\textbf{Case 4.} Suppose $T\neq C_{2},C_{3},C_{4}$. It is
verified by inspection that under the assumptions on the parameters of $E_{T}%
$, $v_{p}\!\left(  \delta_{T,u_{T}}\right)  $ satisfies the given inequality
in Table \ref{ta:conddel}.

The above and Lemma \ref{UpperBoundLemma} allow us to conclude that if
$u_{T}^{-12}\gamma_{T}$ is the minimal discriminant of $E_{T}$, then
$v_{p}\!\left(  N_{T}\right)  \leq v_{p}\!\left(  \delta_{T,u_{T}}\right)  $
for each prime $p$. In particular, $N_{T}\leq\left\vert \delta_{T,u_{T}%
}\right\vert $.
\end{proof}

\section{The Naive Height of \texorpdfstring{$E_{T}$}{}\label{naiveheight}}

Let $E_{T}$ be as defined in Table~\ref{ta:ETmodel} and write $a$ in terms of
$c,d,e$ as given in (\ref{valueofaC34}) for $T=C_{3},C_{4}$. Next, let
\begin{equation}
\left(  \alpha_{T},\beta_{T}\right)  =\left\{
\begin{array}
[c]{cl}%
\left(  \alpha_{T}\!\left(  a,b,d\right)  ,\beta_{T}\!\left(  a,b,d\right)
\right)   & \text{if }T=C_{2},C_{2}\times C_{2},\\
\left(  \alpha_{T}\!\left(  c,d,e,b\right)  ,\beta_{T}\!\left(
c,d,e,b\right)  \right)   & \text{if }T=C_{3},\\
\left(  \alpha_{T}\!\left(  c,d,b\right)  ,\beta_{T}\!\left(  c,d,b\right)
\right)   & \text{if }T=C_{4},\\
\left(  \alpha_{T}\!\left(  a,b\right)  ,\beta_{T}\!\left(  a,b\right)
\right)   & \text{otherwise,}%
\end{array}
\right.  \label{defalpbet}%
\end{equation}
be as defined in \cite[Tables 4 and 5]{2020arXiv200101016B}. These expressions are also found in \cite[definitions.sage]{GitHubMSR}. By \cite[Lemma
2.9 and Theorem 4.4]{2020arXiv200101016B}, the invariants $c_{4}$ and~$c_{6}$
associated to a global minimal model of $E_{T}$ are $u_{T}^{-4}\alpha_{T}$ and
$u_{T}^{-6}\beta_{T}$, respectively. In particular, the naive height of
$E_{T}$ is given by $u_{T}^{-12}\max\!\left\{  \left\vert \alpha_{T}%
^{3}\right\vert ,\beta_{T}^2\right\}  $. In this section, we show that if
$u_{T}^{-12}\gamma_{T}$ is the minimal discriminant of $E_{T}$, then
$\left\vert \delta_{T,u_{T}}\right\vert ^{l_{T}}<u_{T}^{-12}\max\!\left\{
\left\vert \alpha_{T}^{3}\right\vert ,\beta_{T}^{2}\right\}  $ where $l_{T}$
is as given in~(\ref{lTrevisited}).

\begin{lemma}
\label{LemhomPoly}Let $m_{T}$ and $l_{T}$ be as given in (\ref{lTrevisited}).%
\begin{equation}
{\renewcommand{\arraystretch}{1.2}\renewcommand{\arraycolsep}{.12cm}%
\begin{array}
[c]{ccccccccc}\hline
T & C_{2} & C_{2}\times C_{2} & C_{3} & C_{4} & C_{5},C_{6},C_{2}\times C_{4}
& C_{7},C_{8},C_{2}\times C_{6} & C_{9},C_{10} & C_{12},C_{2}\times
C_{8}\\\hline
m_{T} & 6 & 6 & 12 & 12 & 12 & 24 & 36 & 48\\\hline
l_{T} & 1.5 & 2 & 2 & 2.4 & 3 & 4 & 4.5 & 4.8\\\hline
\end{array}
}\label{lTrevisited}%
\end{equation}
Then in the notation of (\ref{delutdef}) and (\ref{defalpbet}), the following
identities hold:%
\begin{align*}
\alpha_{T}  & =\left\{
\begin{array}
[c]{cl}%
a^{\frac{m_{T}}{3}}\alpha_{T}\!\left(  1,1,\frac{b^{2}d}{a^{2}}\right)   &
\text{if }T=C_{2},\\
\left(  c^{3}d^{2}e\right)  ^{\frac{m_{T}}{3}}\alpha_{T}\!\left(
1,1,1,\frac{b}{c^{3}d^{2}e}\right)   & \text{if }T=C_{3},\\
\left(  c^{2}d\right)  ^{\frac{m_{T}}{3}}\alpha_{T}\!\left(  1,1,\frac
{b}{c^{2}d}\right)   & \text{if }T=C_{4},\\
\left(  ad\right)  ^{\frac{m_{T}}{3}}\alpha_{T}\!\left(  1,\frac{b}%
{a},1\right)   & \text{if }T=C_{2}\times C_{2},\\
a^{\frac{m_{T}}{3}}\alpha_{T}\!\left(  1,\frac{b}{a}\right)   &
\text{otherwise,}%
\end{array}
\right.  \\
\beta_{T}  & =\left\{
\begin{array}
[c]{cl}%
a^{\frac{m_{T}}{2}}\beta_{T}\!\left(  1,1,\frac{b^{2}d}{a^{2}}\right)   &
\text{if }T=C_{2},\\
\left(  c^{3}d^{2}e\right)  ^{\frac{m_{T}}{2}}\beta_{T}\!\left(
1,1,1,\frac{b}{c^{3}d^{2}e}\right)   & \text{if }T=C_{3},\\
\left(  c^{2}d\right)  ^{\frac{m_{T}}{2}}\beta_{T}\!\left(  1,1,\frac{b}%
{c^{2}d}\right)   & \text{if }T=C_{4},\\
\left(  ad\right)  ^{\frac{m_{T}}{2}}\beta_{T}\!\left(  1,\frac{b}%
{a},1\right)   & \text{if }T=C_{2}\times C_{2},\\
a^{\frac{m_{T}}{2}}\beta_{T}\!\left(  1,\frac{b}{a}\right)   &
\text{otherwise,}%
\end{array}
\right.  \\
\delta_{T,u_{T}}  & =\left\{
\begin{array}
[c]{cl}%
a^{\frac{m_{T}}{l_{T}}}\!\delta_{T,u_{T}}\!\left(  1,1,\frac{b^{2}d}{a^{2}%
}\right)   & \text{if }T=C_{2},\\
\left(  cde\right)  ^{\frac{m_{T}}{l_{T}}}\delta_{T,u_{T}}\!\left(
1,1,1,\frac{b}{c^{3}d^{2}e}\right)   & \text{if }T=C_{3},\\
\left(  cd\right)  ^{\frac{m_{T}}{l_{T}}}\delta_{T,u_{T}}\!\left(
1,1,\frac{b}{c^{2}d}\right)   & \text{if }T=C_{4},\\
 \left(ad\right)  ^{\frac{m_{T}}{l_{T}}}\delta_{T,u_{T}}\!\left(  1,\frac
{b}{a},1\right)   & \text{if }T=C_{2}\times C_{2},\\
a^{\frac{m_{T}}{l_{T}}}\delta_{T,u_{T}}\!\left(  1,\frac{b}{a}\right)   &
\text{otherwise.}%
\end{array}
\right.
\end{align*}

\end{lemma}
\begin{proof}
This is verified in \cite[Verifications.ipynb]{GitHubMSR}.
\end{proof}

Lemma \ref{LemhomPoly} leads us to the following lemma, which pertains to a
real-valued function and is the key to deducing our main theorem in the next section.

\begin{lemma}
\label{lemmaphi}For each $T$, let $u_{T}$ be as given in \cite[Theorem
4.4]{2020arXiv200101016B}. For $T=C_{4}$, set%
\[
v_{T}=\left\{
\begin{array}
[c]{cl}%
1 & \text{if }u_{T}=c,\\
2 & \text{if }u_{T}=2c.
\end{array}
\right.
\]
Then the real-valued function $\varphi_{T,u_{T}}:%
\mathbb{R}
\rightarrow%
\mathbb{R}
$ defined by%
\[
\varphi_{T,u_{T}}\!\left(  x\right)  =\left\{
\begin{array}
[c]{ll}%
u_{T}^{-12}\max\!\left\{  \left\vert \alpha_{T}\!\left(  1,1,x\right)
\right\vert ^{3},\beta_{T}\!\left(  1,1,x\right)  ^{2}\right\}  -\left\vert
\delta_{T,u_{T}}\!\left(  1,1,x\right)  \right\vert ^{l_{T}} & \text{if
}T=C_{2},\\
\max\!\left\{  \left\vert \alpha_{T}\!\left(  1,1,1,x\right)  \right\vert
^{3},\beta_{T}\!\left(  1,1,1,x\right)  ^{2}\right\}  -\left\vert
\delta_{T,u_{T}}\!\left(  1,1,1,x\right)  \right\vert ^{l_{T}} & \text{if
}T=C_{3},\\
v_{T}^{-12}\max\!\left\{  \left\vert \alpha_{T}\!\left(  1,1,x\right)
\right\vert ^{3},\beta_{T}\!\left(  1,1,x\right)  ^{2}\right\}  -\left\vert
\delta_{T,u_{T}}\!\left(  1,1,x\right)  \right\vert ^{l_{T}} & \text{if
}T=C_{4},\\
u_{T}^{-12}\max\!\left\{  \left\vert \alpha_{T}\!\left(  1,x,1\right)
\right\vert ^{3},\beta_{T}\!\left(  1,x,1\right)  ^{2}\right\}  -\left\vert
\delta_{T,u_{T}}\!\left(  1,x,1\right)  \right\vert ^{l_{T}} & \text{if
}T=C_{2}\times C_{2},\\
u_{T}^{-12}\max\!\left\{  \left\vert \alpha_{T}\!\left(  1,x\right)
\right\vert ^{3},\beta_{T}\!\left(  1,x\right)  ^{2}\right\}  -\left\vert
\delta_{T,u_{T}}\!\left(  1,x\right)  \right\vert ^{l_{T}} & \text{otherwise,}%
\end{array}
\right.
\]
is nonnegative. Moreover, if $\varphi_{T,u_{T}}\!\left(  x\right)  =0$, then
$x$ is irrational.
\end{lemma}

\begin{proof}
This is verified in \cite[Lemma3\_2.nb]{GitHubMSR}. Moreover,
$\varphi_{T,u_{T}}\!\left(  x\right)  =0$ has a solution if $T=C_{2}\times
C_{4}$ with $u_{T}=1,2,4$. In fact, $\left\{  \frac{1}{8}\left(  -3\pm\sqrt
{5}\right)  ,\frac{1}{8}\left(  1\pm\sqrt{5}\right)  \right\}  $ is the
solution set in this case which shows that if $\varphi_{T,u_{T}}\!\left(
x\right)  =0$, then $x$ is irrational.
\end{proof}

\begin{corollary}
\label{maincor}If $u_{T}^{-12}\gamma_{T}$ is the minimal discriminant of
$E_{T}$, then
\[
\left\vert \delta_{T,u_{T}}\right\vert ^{l_{T}}<u_{T}^{-12}\max\!\left\{
\left\vert \alpha_{T}^{3}\right\vert ,\beta_{T}^{2}\right\}  .
\]

\end{corollary}

\begin{proof}
Suppose $T=C_{3}$. Then $u_{T}=c^{2}d$ and by Lemma \ref{LemhomPoly},%
\[
u_{T}^{-12}\max\!\left\{  \left\vert \alpha_{T}^{3}\right\vert ,\beta_{T}%
^{2}\right\}  =\left(  cde\right)  ^{12}\max\!\left\{  \left\vert \alpha
_{T}\!\left(  1,1,1,\frac{b}{c^{3}d^{2}e}\right)  \right\vert ^{3},\beta
_{T}\!\left(  1,1,1,\frac{b}{c^{3}d^{2}e}\right)  ^{2}\right\}  .
\]
It follows by Lemmas \ref{LemhomPoly} and \ref{lemmaphi} that 
\[u_{T}^{-12}
\max\!\left\{  \left\vert \alpha_{T}^{3}\right\vert ,\beta_{T}^{2}\right\}
-\left\vert \delta_{T,u_{T}}\right\vert ^{l_{T}}=\left(  cde\right)
^{12}\varphi_{T,u_{T}}\!\left(  \frac{b}{c^{3}d^{2}e}\right). \]
By Lemma
\ref{lemmaphi}, this is positive which concludes this case.

Now suppose $T=C_{4}$. Then $u_{T}=v_{T}c$ and by Lemma \ref{LemhomPoly},%
\[
u_{T}^{-12}\max\!\left\{  \left\vert \alpha_{T}^{3}\right\vert ,\beta_{T}%
^{2}\right\}  =\left(  cd\right)  ^{12}v_{T}^{-12}\max\!\left\{  \left\vert
\alpha_{T}\!\left(  1,1,\frac{b}{c^{2}d}\right)  \right\vert ^{3},\beta
_{T}\!\left(  1,1,\frac{b}{c^{2}d}\right)  ^{2}\right\}  .
\]
From Lemmas \ref{LemhomPoly} and \ref{lemmaphi}, we deduce that 
\[u_{T}^{-12}\max\!\left\{  \left\vert \alpha_{T}^{3}\right\vert ,\beta_{T}%
^{2}\right\}  -\left\vert \delta_{T,u_{T}}\right\vert ^{l_{T}}=\left(
cd\right)  ^{12}\varphi_{T,u_{T}}\!\left(  \frac{b}{c^{2}d}\right).\]
Lemma
\ref{lemmaphi} now implies the desired inequality.

For $T\neq C_{3},C_{4}$, a similar argument shows via Lemmas \ref{LemhomPoly}
and \ref{lemmaphi} that
\[
u_{T}^{-12}\max\!\left\{  \left\vert \alpha_{T}^{3}\right\vert ,\beta_{T}%
^{2}\right\}  -\left\vert \delta_{T,u_{T}}\right\vert ^{l_{T}}=\left\{
\begin{array}
[c]{cl}%
a^{m_{T}}\varphi_{T,u_{T}}\!\left(  \frac{b^{2}d}{a^{2}}\right)  & \text{if
}T=C_{2},\\
\left(  ad\right)  ^{m_{T}}\varphi_{T,u_{T}}\!\left(  \frac{b}{a}\right)  &
\text{if }T=C_{2}\times C_{2},\\
a^{m_{T}}\varphi_{T,u_{T}}\!\left(  \frac{b}{a}\right)  & \text{if }T\neq
C_{2},C_{3},C_{4},C_{2}\times C_{2}.
\end{array}
\right.
\]
By Lemma \ref{lemmaphi}, we conclude that $\left\vert \delta_{T,u_{T}%
}\right\vert ^{l_{T}}<u_{T}^{-12}\max\!\left\{  \left\vert \alpha_{T}%
^{3}\right\vert ,\beta_{T}^{2}\right\}  $.
\end{proof}

\section{Lower Bounds for the Modified Szpiro Ratio\label{lowerboundsec}}
With the results established in the previous section, we are now ready to prove Theorem~\ref{Thm1}. We first show that $l_T$ is a lower bound for the modified Szpiro ratio, and consider the sharpness of $l_T$ in Theorem~\ref{thmshp}.
\begin{theorem}
\label{mainthm}Let $T$ be one of the fifteen torsion subgroups allowed by
Theorem \ref{MazurTorThm}. If $E/%
\mathbb{Q}
$ is an elliptic curve with $T\hookrightarrow E\!\left(
\mathbb{Q}
\right)  _{\text{tors}}$, then $\sigma_{m}\!\left(  E\right)  >l_{T}$ where
\[
{\renewcommand{\arraystretch}{1.2}\renewcommand{\arraycolsep}{.132cm}%
\begin{array}
[c]{ccccccccc}\hline
T & C_{1} & C_{2} & C_{3},C_{2}\times C_{2} & C_{4} & C_{5},C_{6},C_{2}\times
C_{4} & C_{7},C_{8},C_{2}\times C_{6} & C_{9},C_{10} & C_{12},C_{2}\times
C_{8}\\\hline
l_{T} & 1 & 1.5 & 2 & 2.4 & 3 & 4 & 4.5 & 4.8\\\hline
\end{array}
}%
\]

\end{theorem}

\begin{proof}
Let $E/%
\mathbb{Q}
$ be an elliptic curve and let $c_{4}$ and $c_{6}$ be the invariants
associated to a global minimal model of $E$. Let $\Delta$ denote the minimal discriminant of $E$.
Then the identity $1728\Delta
=c_{4}^{3}-c_{6}^{2}$,
implies that $\left\vert \Delta\right\vert <\max\!\left\{  \left\vert
c_{4}^{3}\right\vert ,c_{6}^{2}\right\}  $ and thus $\sigma_{m}\!\left(
E\right)  >1$. Now suppose $E\!\left(
\mathbb{Q}
\right)  _{\text{tors}}\cong T$ where $T$ is non-trivial. Suppose further that
if $T=C_{3}$, then the $j$-invariant of $E$ is non-zero. Then Proposition
\ref{rationalmodels} implies that $E$ is $%
\mathbb{Q}
$-isomorphic to $E_{T}$ for some integers $a,b,$ and $d$. By \cite[Lemma 2.9
and Theorem 4.4]{2020arXiv200101016B}, the invariants $c_{4}$ and $c_{6}$
associated to a global minimal model of $E$ are $u_{T}^{-4}\alpha_{T}$ and
$u_{T}^{-6}\beta_{T}$, respectively. Next, let $N_{E}$ denote the conductor of
$E$. By Proposition~\ref{ch:enh:bdcon} and Corollary~\ref{maincor}, we have
that%
\[
N_{E}\leq\left\vert \delta_{T,u_{T}}\right\vert ^{l_{T}}<u_{T}^{-12}%
\max\!\left\{  \left\vert \alpha_{T}^{3}\right\vert ,\beta_{T}^{2}\right\}
=\max\!\left\{  \left\vert c_{4}^{3}\right\vert ,c_{6}^{2}\right\}  .
\]
Taking logarithms results in
\[
\log N_{E}\leq l_{T}\log\left\vert \delta_{T,u_{T}}\right\vert <\log\max\!\left\{
\left\vert c_{4}^{3}\right\vert ,c_{6}^{2}\right\}  \qquad\Longrightarrow
\qquad\frac{l_{T}\log\left\vert \delta_{T,u_{T}}\right\vert }{\log N_{E}%
}<\sigma_{m}\!\left(  E\right)  .
\]
By Proposition \ref{ch:enh:bdcon}, $\frac{\log\left\vert \delta_{T,u_{T}%
}\right\vert }{\log N_{E}}\geq1$ and hence $l_{T}<\sigma_{m}\!\left(
E\right)  $.

It remains to show the case when $E\!\left(
\mathbb{Q}
\right)  _{\text{tors}}\cong C_{3}$ with $j$-invariant $0$. If this is the
case, then Proposition \ref{rationalmodels} implies that $E$ is $%
\mathbb{Q}
$-isomorphic to either $E_{C_{3}}\!\left(  24,1\right)  $ or $E_{C_{3}^{0}%
}:y^{2}+ay=x^{3}$ for some positive cubefree integer $a$. By the above,
$\sigma_{m}\!\left(  E_{C_{3}}\!\left(  24,1\right)  \right)  >2$ and so it
suffices to consider the case when $E$ is $%
\mathbb{Q}
$-isomorphic to $E_{C_{3}^{0}}$. By \cite[Theorem 4.4]{2020arXiv200101016B},
the minimal discriminant of $E$ is $-27a^{4}$ and the invariants $c_{4}$ and
$c_{6}$ associated to a global minimal model of $E$ are $0$ and $-216a^{2}$,
respectively. Now let $\mu=27a^{2}$. By \cite[Theorem 3.4]{BarRoy},%
\[
N_{E}=3^{f_{p}}\prod_{p|a,\ p\neq3}p^{2}\qquad\text{where }f_{p}\leq\left\{
\begin{array}
[c]{cl}%
3 & \text{if }v_{3}\!\left(  a\right)  =0,\\
5 & \text{if }v_{3}\!\left(  a\right)  =1,2.
\end{array}
\right.
\]
Hence $N_{E}\leq\mu$, and the proof now follows since
\[
N_{E}\leq\mu^{2}<c_{6}^{2}\qquad\Longrightarrow\qquad\frac{2\log\mu}{\log
N_{E}}<\sigma_{m}\!\left(  E\right)  .\qedhere
\]
\end{proof}

It remains to show the sharpness of the lower bounds $l_{T}$. To this end, we first consider the following lemma, which will aid us in establishing that the lower bounds are sharp. We note that the first part of the proof is identical to that of~\cite[Lemma 2.14 (1)]{MR3479215}.

\begin{lemma}\label{polysquarefree}
Let $f\!\left(  x\right)  $ be a squarefree polynomial with content $1$ such
that each of the irreducible factors of $f\!\left(  x\right)  $ have degree
at most $3$. If there exists an $a\in%
\mathbb{Z}
$ such that $f\!\left(  a\right)  $ is squarefree, then
$\left\{  x\in%
\mathbb{Z}
\mid f\!\left(  x\right)  \text{ is squarefree}\right\}  $ is infinite.
\end{lemma}

\begin{proof}
First, suppose that there is no prime $p$ such that $p|f\!\left(  x\right)  $
for each $x\in%
\mathbb{Z}
$. If there exists an $a\in%
\mathbb{Z}
$ such that $p\nmid f\!\left(  a\right)  $ for each prime $p\leq\deg f$, then
it was shown in \cite[Theorem 1.2]{MR3533307} that there is a positive
constant $c=c\!\left(  f\right)  $ such that%
\[
\#\left\{  x\in%
\mathbb{Z}
\cap\left[  1,X\right]  \mid f\!\left(  x\right)  \text{ is squarefree}%
\right\}  =\left(  c+o\!\left(  1\right)  \right)  X\qquad\text{as
}X\rightarrow\infty.
\]
In particular, $\left\{  x\in%
\mathbb{Z}
\mid f\!\left(  x\right)  \text{ is squarefree}\right\}  $ is infinite.

Now suppose that there is an $a\in%
\mathbb{Z}
$ such that $f\!\left(  a\right)  $ is squarefree. Let%
\[
q_{1}=\prod_{\substack{p|f\!\left(  x\right)  \ \forall x\in%
\mathbb{Z}
,\\p|f\!\left(  a\right)  \ \text{for }p\leq\deg f}}p\qquad\text{and}\qquad
q_{2}=\prod_{\substack{p\nmid f\!\left(  a\right)  \\\text{for }p\leq\deg
f}}p.
\]
Then $g\!\left(  x\right)  =\frac{1}{q_{1}}f\!\left(  a+q_{1}^{2}q_{2}%
^{2}x\right)  \in%
\mathbb{Z}
\!\left[  x\right]  $ and by construction, $g\!\left(  x\right)  $ has the
property that there is no prime $p$ such that $p|g\!\left(  x\right)  $ for
each $x\in%
\mathbb{Z}
$. We also have that $p\nmid g\!\left(  a\right)  $ for each prime $p$ with
$p\leq\deg g$. By \cite[Theorem 1.2]{MR3533307}, $\left\{  x\in%
\mathbb{Z}
\mid g\!\left(  x\right)  \text{ is squarefree}\right\}  $ is infinite. Moreover, if $p$ is a prime such that $p|q_{1}$, then $p\nmid g\!\left(  x\right)
$ for each $x\in%
\mathbb{Z}
$. The result now follows since%
\begin{align*}
\left\{  x\in%
\mathbb{Z}
\mid g\!\left(  x\right)  \text{ is squarefree}\right\}    & =\left\{  x\in%
\mathbb{Z}
\mid f\!\left(  a+q_{1}^{2}q_{2}^{2}x\right)  \text{ is squarefree}\right\}
\\
& \subseteq\left\{  x\in%
\mathbb{Z}
\mid f\!\left(  x\right)  \text{ is squarefree}\right\}  . \qedhere
\end{align*}

\end{proof}

\begin{theorem}
\label{thmshp}The lower bounds $l_{T}$ in Theorem \ref{mainthm} are sharp.
\end{theorem}

\begin{proof}
Let $n$ be an integer and define the elliptic curve $F_{C_{1}}=F_{C_{1}%
}\!\left(  n\right)  $ where $F_{C_{1}}:y^{2}+y=x^{3}+\left(  3n+1\right)  x$.
For $T\neq C_{1}$, let $A_{T},$ $B_{T}$, and $D_{T}$ be as given below.%
\[
{\renewcommand{\arraystretch}{1.2}\renewcommand{\arraycolsep}{.1cm}%
\begin{array}
[c]{ccccccccccc}\hline
T & C_{2} & C_{3} & C_{4} & C_{5},C_{9},C_{2}\times C_{4} & C_{6},C_{7} &
C_{8},C_{10} & C_{12} & C_{2}\times C_{2} & C_{2}\times C_{6} & C_{2}\times
C_{8}\\\hline
A_{T} & -1 & 1 & 256n^{2} & 2n+1 & 3n+1 & 4n+1 & 6n+1 & 16n & 8n+3 &
4n\\\hline
B_{T} & 8 & n & 4n^{2}-1 & n & n & n & n & 4n+1 & -1 & n+1\\\hline
D_{T} & n &  &  &  &  &  &  & 1 &  & \\\hline
\end{array}
}%
\]
For $T\not =C_{1},C_{2},C_{2}\times C_{2}$, let $F_{T}\!\left(  n\right)
=E_{T}\!\left(  A_{T},B_{T}\right)  $. For $T=C_{2},C_{2}\times C_{2}$, let
$F_{T}\!\left(  n\right)  =E_{T}\!\left(  A_{T},B_{T},D_{T}\right)  $. It follows that $T\hookrightarrow F_{T}\!\left(  n\right)\!\left(\mathbb{Q}\right)  $ for
each $n$.

From the definition of $F_{T}\!\left(  n\right)  $ for $T\neq C_{1}$, we
deduce by \cite[Corollary 7.2]{2020arXiv200101016B} that $F_{T}$ is semistable
for each integer $n$. For $T=C_{1}$, observe that the invariant $c_{4}$ and
$c_{6}$ of $F_{C_{1}}$ are $-48\left(  3n+1\right)  $ and $-216$,
respectively. Since $c_{6}$ is six-powerfree, we conclude that $F_{C_{1}}$ is
given by a global minimal model. In particular, the minimal discriminant
$\mathfrak{D}_{C_{1}}$ of $F_{C_{1}}$ is the discriminant of $F_{C_{1}}$.
Since $\mathfrak{D}_{C_{1}}=-\left(  12n+7\right)  \left(  144n^{2}%
+60n+13\right)  $, we see that $\gcd\!\left(  216,\mathfrak{D}_{C_{1}%
}\right)  =1$. From this, we establish that $F_{C_{1}}$ is semistable for each
integer $n$.

Now let $\mathfrak{D}_{T}$ denote the discriminant of $F_{T}$ and denote the
invariants $c_{4}$ and $c_{6}$ of $F_{T}$ by $\mathfrak{A}_{T}$ and
$\mathfrak{B}_{T}$, respectively. By the above and \cite[Theorem
4.4]{2020arXiv200101016B} for $T\neq C_{1}$, we have that the minimal
discriminant of $F_{T}$ is $w_{T}^{-12}\mathfrak{D}_{T}$ where
\[
{\renewcommand{\arraystretch}{1.2}\renewcommand{\arraycolsep}{.25cm}%
\begin{array}
[c]{ccccccccc}\hline
T & C_{4} & C_{2}\times C_{8} & C_{2}\times C_{6} & C_{2},C_{2}\times C_{2} &
\text{otherwise} &  &  & \\\hline
w_{T} & 32n & 64 & 16 & 2 & 1 &  &  & \\\hline
\end{array}
}%
\]
In particular, the invariants $c_{4,T}$ and $c_{6,T}$ associated to a global
minimal model of $F_{T}$ are $w_{T}^{-4}\mathfrak{A}_{T}$ and $w_{T}%
^{-6}\mathfrak{B}_{T}$, respectively. By \cite[Theorem4\_3.nb]{GitHubMSR}, we have that for each integer $n$ with $\left\vert n\right\vert >1$ it is the case that
\[
\mathfrak{h}_{T}=\max\!\left\{  \left\vert c_{4,T}^{3}\right\vert ,c_{6,T}%
^{2}\right\}  =\left\{
\begin{array}
[c]{cl}%
c_{6,T}^{2} & \text{if }T=C_{3},C_{5},C_{7},C_{8},C_{9},\\
\left\vert c_{4,T}^{3}\right\vert  & \text{otherwise.}%
\end{array}
\right.
\]
In Table \ref{ta:sharp}, we give the expression for $\mathfrak{h}_{T}\!\left(
n\right)  $. Next, let $\mathfrak{f}_{T}\!\left(  n\right)  $ be as given in
Table \ref{ta:sharp}. Symbolic expressions for $\mathfrak{h}_{T}\!\left(
n\right)  $ and $\mathfrak{f}_{T}\!\left(
n\right)  $ are found in \cite[definitions.sage]{GitHubMSR}.
\vspace{-0.5em} {\renewcommand*{\arraystretch}{1.27}%
\begin{longtable}{cC{3.3in}C{1.6in}}
\caption{The Quantities $\mathfrak{h}_{T}\!\left(  n\right)  $ and $\mathfrak{f}_{T}\!\left(
n\right)  $ }
		\label{ta:sharp}	\\
\hline
$T$ & $\mathfrak{h}_{T}\!\left(  n\right)  $ & $\mathfrak{f}_{T}\!\left(
n\right)  $\\
\hline
\endfirsthead
\caption[]{\emph{continued}}\\
\hline
$T$ & $\mathfrak{h}_{T}\!\left(  n\right)  $ & $\mathfrak{f}_{T}\!\left(
n\right)  $\\
\hline
\endhead
\hline
\multicolumn{3}{r}{\emph{continued on next page}}
\endfoot
\hline
\endlastfoot
	
$C_{1}$ & $|144n+48|^{3}$ & $(12n+7)(144n^{2}+60n+13)$\\\hline
$C_{2}$ & $| 192n+1| ^{3}$ & $n(64n-1)$\\\hline
$C_{3}$ & $(216n^{2}-36n+1)^{2}$ & $n(27n-1)$\\\hline
$C_{4}$ & $|5136n^{4}-264n^{2}+1|^{3}$ & $n(2n-1)(2n+1)(20n^{2}-1)$\\\hline
$C_{5}$ & $(421n^{4}+526n^{3}+206n^{2}+26n+1)^{2}(5n^{2}+4n+1)^{2}$ &
$n(2n+1)(25n^{2}+15n+1)$\\\hline
$C_{6}$ & $|120n^{3}+84n^{2}+18n+1|^{3}|6 n+1|^{3}$ &
$n(12n+1)(3n+1)(4n+1)$\\\hline
$C_{7}$ & $(577801n^{12}+2519622n^{11}+4989285n^{10}+5920782n^{9}%
+4680102n^{8}+2590434n^{7}+1027173n^{6}+293286n^{5}+59682n^{4}+8414n^{3}%
+777n^{2}+42n+1)^{2}$ & $n(2n+1)(3n+1)(49n^{3}+49n^{2}+14n+1)$\\\hline
$C_{8}$ & $(2696n^{8}+5984n^{7}+5424n^{6}+2272n^{5}+184n^{4}-224n^{3}%
-96n^{2}-16n-1)^{2}(56n^{4}+16n^{3}-16n^{2}-8n-1)^{2}$ &
$n(4n+1)(2n+1)(3n+1)(8n^{2}-1)$\\\hline
$C_{9}$ & $(22329n^{18}+242514n^{17}+1250883n^{16}+4061502n^{15}%
+9272961n^{14}+15760494n^{13}+20613420n^{12}+21173562n^{11}+17291556n^{10}%
+11299356n^{9}+5916807n^{8}+2474496n^{7}+819423n^{6}+211626n^{5}%
+41607n^{4}+5994n^{3}+594n^{2}+36n+1)^{2}$ & $n(n+1)(2n+1)(3n^{2}%
+3n+1)(9n^{3}+18n^{2}+9n+1)$\\\hline
$C_{10}$ & $|635920n^{12}+2733440n^{11}+5299680n^{10}+6129200n^{9}%
+4710480n^{8}+2534880n^{7}+979520n^{6}+273840n^{5}+54960n^{4}+7720n^{3}%
+720n^{2}+40n+1|^{3}$ & $n(2n+1)(4n+1)(3n+1)(20n^{2}+10n+1)(5n^{2}+5n+1)$\\\hline
$C_{12}$ & $|42787896n^{12}+129338064n^{11}+173452752n^{10}+137824296n^{9}%
+72709428n^{8}+26936592n^{7}+7205496n^{6}+1405032n^{5}+198498n^{4}%
+19836n^{3}+1332n^{2}+54n+1|^{3}|366n^{4}+348n^{3}+120n^{2}+18n+1|^{3}$ &
$n(6n+1)(4n+1)(5n+1)(6n^{2}+6n+1)(26n^{2}+10n+1)(21n^{2}+9n+1)$\\\hline
$C_{2}\times C_{2}$ & $|208n^{2}-8n+1|^{3}$ & $n(4n+1)(12n-1)$\\\hline
$C_{2}\times C_{4}$ & $|976n^{4}+672n^{3}+200n^{2}+24n+1|^{3}$ &
$n(2n+1)(10n+1)(6n+1)$\\\hline
$C_{2}\times C_{6}$ & $|439104n^{6}+1005408n^{5}+958080n^{4}+486360n^{3}%
+138720n^{2}+21078n+1333|^{3}|84n^{2}+66n+13|^{3}$ &
$(3n+1)(12n+5)(18n+7)(2n+1)(5n+2)(8n+3)$\\\hline
$C_{2}\times C_{8}$ & $|51361n^{16}+180064n^{15}+301720n^{14}+511840n^{13}%
+1140780n^{12}+2129632n^{11}+2812328n^{10}+2658400n^{9}+1853894n^{8}%
+973088n^{7}+387560n^{6}+116768n^{5}+26220n^{4}+4256n^{3}+472n^{2}+32n+1|^{3}$
& $n(n+1)(2n+1)(3n+1)(n^{2}-2n-1)(7n^{2}+6n+1)(5n^{2}+4n+1)$
\end{longtable}}

\noindent It is then verified that $\operatorname{rad}\!\left(  w_{T}%
^{-12}\mathfrak{D}_{T}\right)  =\operatorname{rad}\!\left(  \mathfrak{f}%
_{T}\!\left(  n\right)  \right)  $. Now define%
\[
S_{T}=\left\{  n\in%
\mathbb{Z}
\mid\left\vert n\right\vert >1,\ \mathfrak{f}_{T}\!\left(  n\right)  \text{ is
squarefree}\right\}  .
\]
Observe that if $n\in S_{T}$, then the conductor $N_{F_{T}}$ of $F_{T}$
satisfies $N_{F_{T}}=\left\vert \mathfrak{f}_{T}\!\left(  n\right)
\right\vert $ since $F_{T}$ is semistable. In \cite[Theorem4\_3.ipynb]{GitHubMSR}, it is shown that for each $T$, $\mathfrak{f}_{T}\!\left(  n\right)$ has content~$1$ and that there exists an $a \in \mathbb{Z}$ such that $\mathfrak{f}_{T}\!\left(  a\right)$ is squarefree. Moreover, each $\mathfrak{f}_{T}\!\left(  n\right)$ has the property that each of its irreducible factors has degree at most $3$. Consequently, $S_T$ is infinite for each $T$ by Lemma~\ref{polysquarefree}.

From \cite[Theorem4\_3.ipynb]{GitHubMSR}, we deduce that
\[
\lim_{n\in S_{T},\ \left\vert n\right\vert \rightarrow\infty}\sigma
_{m}\!\left(  F_{T}\!\left(  n\right)  \right)  =\lim_{n\in S_{T},\ \left\vert
n\right\vert \rightarrow\infty}\frac{\log\mathfrak{h}_{T}}{\log\left\vert
\mathfrak{f}_{T}\!\left(  n\right)  \right\vert }=l_{T}.
\]
In particular, we have that the lower bound in\ Theorem \ref{mainthm} is sharp
since the above shows that there is an infinite sequence of elliptic curves
$F_{T}\!\left(  n\right)  $ with $T\hookrightarrow F_{T}\!\left(  n\right) \!\left(\mathbb{Q}\right)  $
with the property that the modified Szpiro ratio of $F_{T}\!\left(  n\right)
$ tends to $l_{T}$ as $\left\vert n\right\vert \rightarrow\infty$.
\end{proof}

We conclude with the following application of Corollary~\ref{enhMSRtrivial}:

\begin{corollary}\label{lastcor}
There are infinitely many $n\in%
\mathbb{Z}
$ for which the elliptic curve $F_{C_{1}}\!\left(  n\right)  :y^{2}%
+y=x^{3}+\left(  3n+1\right)  x$ has trivial torsion subgroup and global
Tamagawa number equal to $1$.
\end{corollary}

\begin{proof}
Assume the notation in the proof of Theorem \ref{thmshp}. The proof
established that $F_{C_{1}}\!\left(  n\right)  $ is semistable for each $n\in%
\mathbb{Z}
$ and that $S_{C_{1}}=\left\{  n\in%
\mathbb{Z}
\mid\mathfrak{f}_{C_{1}}\!\left(  n\right)  \ \text{is squarefree}\right\}  $
is infinite. For $n\in S_{C_{1}}$,%
\[
\sigma_{m}\!\left(  F_{C_{1}}\!\left(  n\right)  \right)  =\frac
{\log\mathfrak{h}_{C_{1}}\!\left(  n\right)  }{\log\left\vert \mathfrak{f}%
_{C_{1}}\!\left(  n\right)  \right\vert }.
\]
By \cite[Corollary4\_4.nb]{GitHubMSR}, we have that if $\left\vert
n\right\vert \geq13$, then $\mathfrak{h}_{C_{1}}\!\left(  n\right)
<\mathfrak{f}_{C_{1}}\!\left(  n\right)  ^{1.5}$. In particular, if $n\in
S_{C_{1}}$ with $\left\vert n\right\vert \geq13$, then $\sigma_{m}\!\left(
F_{C_{1}}\!\left(  n\right)  \right)  <1.5$. By Corollary \ref{enhMSRtrivial},
it follows that the torsion subgroup of $F_{C_{1}}\!\left(  n\right)  $ is trivial.
Lastly, the minimal discriminant of $F_{C_{1}}\!\left(  n\right)  $ is
$\mathfrak{f}_{C_{1}}\!\left(  n\right)  $. Since $\mathfrak{f}_{C_{1}%
}\!\left(  n\right)  $ is squarefree, we conclude by  Tate's Algorithm~\cite{MR0393039} that for each prime $p$, the local Tamagawa number at $p$ of $F_{C_{1}}\!\left(
n\right)  $ is equal to $1$. In particular, the
global Tamagawa number of $F_{C_{1}}\!\left(
n\right)  $ is equal to $1$.
\end{proof}

\noindent \textbf{Acknowledgments.} The author thanks Chung Pang Mok for suggesting the investigation of \cite[Theorem 4]{MR1635726} in the context of the modified Szpiro conjecture and Edray Goins and Manami Roy for helpful comments in the preparation of this article. In the original submission, \cite{MR1654759} was used to show the lower bounds' sharpness in Theorem~\ref{Thm1} under the assumption of the $ABC$ Conjecture. The author is grateful to the referee for their comments and suggestion of using \cite{MR3533307} in place of \cite{MR1654759}. Lastly, the author thanks Andrew Booker for his helpful comments, which helped simplify the proof of Lemma~\ref{polysquarefree}.

\bibliographystyle{amsplain}
\bibliography{bibliography}

\end{document}